\documentclass[12pt]{article}
\usepackage{amsmath,amssymb,amsbsy,amsfonts,amsthm,latexsym,
                amsopn,amstext,amsxtra,euscript,amscd,hyperref,url}
                
\usepackage[utf8]{inputenc} 

\usepackage{geometry} 
\usepackage{graphicx} 

\usepackage{booktabs} 
\usepackage{array} 
\usepackage{paralist} 
\usepackage{verbatim} 
\usepackage{subfig} 
\usepackage{amsfonts}
\usepackage{amssymb}
\usepackage{amsmath}
\usepackage{eufrak}

\usepackage{fancyhdr} 
\newtheorem{thm}{Theorem}
\newtheorem{lem}[thm]{Lemma}
\newtheorem{lemma}[thm]{Lemma}

\newtheorem{theorem}[thm]{Theorem}

\newtheorem{corollary}[thm]{Corollary}

\def\mand{\qquad\mbox{and}\qquad}

\def\bbbf{{\mathrm I\!F}}

\def\bbbf{{\mathbb F}}
\def \F{{\bbbf}}

\def\ordp{{\,\mathrm{ord}_p}\,}

\def\\{\cr}
\def\({\left(}
\def\){\right)}
\def\[{\left[}
\def\]{\right]}
\def\<{\langle}
\def\>{\rangle}

\def\cA{{\mathcal A}}
\def\cB{{\mathcal B}}
\def\cC{{\mathcal C}}
\def\cD{{\mathcal D}}
\def\cE{{\mathcal E}}

\def\cX{{\mathcal X}}
\def\cY{{\mathcal Y}}

\begin{document}
\date{ }
\author{
{\sc   Bryce Kerr} \\
{Department of Computing, Macquarie University} \\
{Sydney, NSW 2109, Australia} \\
{\tt  bryce.kerr@mq.edu.au}}
\title{ Incomplete exponential sums over exponential functions}

 \maketitle

\begin{abstract}
We extend some methods of bounding exponential sums of the type $\displaystyle\sum_{n\le N}e^{2\pi iag^n/p}$ to deal with the case when $g$ is not necessarily a primitive root. We also show some recent results of Shkredov concerning additive properties of multiplicative subgroups  imply new bounds for the sums under consideration.
\end{abstract}
\section{Introduction}
For $p$ prime, $g \in \F_p^{*}$ of order $t$ and integer $N\leq t$ we consider the sums
\begin{equation}
\label{SUM}
S_{g,p}(\lambda,N)=\displaystyle\sum_{n=1}^{N}e_p(\lambda g^n)
\end{equation}
where $e_p(z)=e^{2\pi i z/p}$ and $\gcd(\lambda,p)=1$. Estimates for $S_{g,p}(\lambda,N)$ have been considered in a number of works. For instance Korobov~\cite{Kor} obtains the bound
\begin{equation}
\max_{\gcd(\lambda,p)=1}|S_{g,p}(\lambda,N)|\ll p^{1/2}\log{p}
\end{equation}
which is used to study the distribution of digits in decimal exansions of rational numbers (see also \cite{ShpSte} and references therein). If $g$ is a primitive root, Bourgain and Garaev \cite{BouGar}  give bounds for the number of solutions to the equation
$$g^{x_1}+g^{x_2}\equiv g^{x_3}+g^{x_4} \pmod p, \quad 1\le x_1,\dots,x_4 \le N,$$
which they use to estimate $S_{g,p}(\lambda,N)$. Konyagin and Shparlinski~\cite{KoSh} improve on this bound and give applications to the gaps between powers of a primitive root.\newline \indent The case of complete sums with $N=t$ have also been considered by a number of authors (see for example \cite{HBK}) from which corresponding bounds for the incomplete sums can be obtained using a method of~\cite{Sh}.
\newline
\indent
  We show that the proof of~\cite[Theorem 1.4]{BouGar} can be generalized to deal with the case when $g$ is not a primitive root. This gives an upper bound for the sums
\begin{equation}
\label{4th}
\sum_{\lambda \in  \F_p^*}\left|S_{g,p}(\lambda,N)\right|^4.
\end{equation}
We then combine the argument of~\cite[Theorem 1]{KoSh} and our upper bound for~\eqref{4th} to deduce a bound for $S_{g,p}(\lambda,N)$. Next we show that~\cite[Theorem 34]{Sh} combined with a method of \cite{Sh} gives another bound for $S_{g,p}(\lambda,N)$. \newline \indent We use the notation $f(x) \ll g(x)$ and $f(x)=O(g(x))$ to mean there exists some absolutle constant $C$ such that $f(x)\le Cg(x)$ and $f(x)=o(g(x))$ will mean that $f(x)\le \varepsilon g(x)$ for any $\varepsilon>0$ and sufficiently large $x$.
\section{Main results}

\begin{theorem}
\label{lem:Bound SN4}
For prime $p$  and $g\in \F_p^*$ of order $t$ and integer $N\le t$, we have
$$
\sum_{\lambda \in  \F_p^*}\left|S_{g,p}(\lambda,N)\right|^4 \ll p N^{71/24+o(1)} \(1 + (N^2/t)^{1/24}\)
$$
as $N\rightarrow \infty$.
\end{theorem}

We use Theorem~\ref{lem:Bound SN4} to deduce

\begin{theorem}
\label{KoSh1}
For $g\in \F_p^*$ of order $t$ and integer $N\le t$, we have
$$\max_{\gcd(\lambda,p)=1}|S_{g,p}(\lambda,N)| \leq \begin{cases}p^{1/8}N^{71/96+o(1)}, \quad \quad  \quad \   \ \ \  N\le t^{1/2},
\\ p^{1/8}t^{-1/96}N^{73/96+o(1)}, \quad   \ \ \ t^{1/2}< N\le p^{1/2},
\\p^{1/4}t^{-1/96}N^{49/96+o(1)}, \quad  \ \ \ p^{1/2}<N < t, \end{cases}$$
as $N \rightarrow \infty$.
\end{theorem}
The following is a consequence of~\cite[Lemma 2]{Kor} and~\cite[Theorem~34]{Sk} 
\begin{theorem}
\label{sk}
For $g\in \F_p^*$ of order $t$ and integer $N\le t$, we have
$$\max_{\gcd(\lambda,p)=1}|S_{g,p}(\lambda,N)| \ll \begin{cases}  p^{1/8}t^{22/36}(\log{p})^{7/6}, \quad  \ \  \  \ t \le p^{1/2}, \\  
 p^{1/4}t^{13/36}(\log{p})^{7/6}, \quad \  \ p^{1/2}<t \le p^{3/5}(\log{p})^{-6/5}, \\ 
 p^{1/6}t^{1/2}(\log{p})^{4/3}, \quad \ \ \ \ \  p^{3/5}<t \le p^{2/3}(\log{p})^{-2/3}, \\
p^{1/2}\log{p}, \quad \quad \quad  \ \ \ \ \  \ \ \  \ \ t>p^{2/3}(\log{p})^{-2/3}. \end{cases}$$
\end{theorem}
We may combine Theorem~\ref{KoSh1} and Theorem~\ref{sk} into a single result for particular values of $t$. For instance, when $t$ has order $p^{1/2}$ we get
\begin{corollary}
Suppose $g \in \F_p^*$  has order $t$ with $p^{1/2}\ll t \ll p^{1/2}$. Then
$$\max_{\gcd(\lambda,p)=1}|S_{g,p}(\lambda,N)|\le \begin{cases} p^{1/8+o(1)}N^{71/96}, \quad N\le p^{1/4}, \\ p^{23/192+o(1)}N^{73/96}, \quad p^{1/4}< N \le p^{179/438}, \\ p^{31/72+o(1)}, \quad  p^{179/438}<N \ll p^{1/2}. \end{cases}$$
\end{corollary}

\section{Preliminary Results}
Given $\cA,\cB \subseteq \F_p$ we define
$$\cA+\cB=\{ a+b \ : \ a \in \cA, \  b \in \cB \}$$
and 
$$\frac{\cA}{\cB}=\{ ab^{-1} \ : \ a\in \cA, \ b \in \cB, \  b \neq 0 \}.$$
We follow the method of~\cite{BouGar} to generalise~\cite[Lemma~2.8]{BouGar}
\begin{lem}
\label{sum}
Suppose $g \in \F_p^{*}$ has multiplicative order $t$ and let $L_1,L_2,M$ be nonnegative integers with $1\leq M\leq t$. Let 
$$\cX\subseteq [L_1+1,L_1+M] \mand \cY\subseteq [L_2+1,L_2+M]
$$ 
be two sets of integers of cardinalities 
$$\# \cX =M\Delta_1 \mand \# \cY =M\Delta_2.
$$
Then for the sets
$$ \cA=\{g^x~:~x\in \cX \} \mand \cB=\{g^y~:~y\in \cY\}
$$
we have 
$$\#(\cA+\cB)\ge \min\left \{M^{9/8+o(1)}\Delta_1^{3/4}\Delta_2, 
t^{1/8}M^{7/8+o(1)}\Delta_1^{5/8}\Delta_2 \right \}.
$$
\end{lem}

\begin{proof}
We follow the proof of~\cite[Lemma~2.8]{BouGar} and begin by considering the sum
$$\sum_{a_1, a_2 \in \cA}\#\(a_1\cB \cap a_2 \cB\)= 
\# \{ (a_1,a_2,b_1,b_2)\in \cA \times \cA \times \cB \times \cB~:~a_1b_1=a_2b_2 \}.
$$
By~\cite[Lemma~2.9]{TaoVu} and the Cauchy-Schwarz inequality we get 
$$\sum_{a_1\,a_2 \in \cA}\#\(a_1\cB \cap a_2 \cB\)\ge 
\frac{\(\# \cA\)^2\(\# \cB\)^2}{\#(\cA \cB)}$$
hence there exists some fixed $a_0\in \cA$ such that
$$\sum_{a\in \cA}\#\(a\cB \cap a_0 \cB\)\ge \frac{\#\cA\(\# \cB\)^2}{\#(\cA \cB)}.$$
Using an argument from~\cite[Theorem~1]{Gar}, for positive integer $j\le \log\#\cB/\log{2}+1$, let $\cD_j$ be the set of all $a\in \cA$ such that
$$2^{j-1}\le \#\(a\cB \cap a_0 \cB\)< 2^{j}$$
and set $\cD_j=\emptyset$ otherwise.
Then we have
$$\sum_{j}\sum_{a\in \cD_j}2^j \ge \sum_{a\in \cA}\#\(a\cB \cap a_0 \cB\)\ge
 \frac{\#\cA\(\# \cB\)^2}{\#(\cA \cB)}.$$
We choose $j_0$ so that  $\sum_{a\in \cD_{j}}2^j$ is maximum for $j=j_0$ and let
\begin{equation}
\label{A_1 N def}
N=2^{j_0-1}, \qquad \cA_1=\cD_{j_0}\subseteq \cA,
\end{equation}
so that 
\begin{equation}
\label{B and a}
 N\leq   \#\(a\cB \cap a_0 \cB\) \leq 2N .
\end{equation}
We have
$$\left(\log\#\cB/\log{2}+1\right)\sum_{a\in \cD_{j_0}}2^{j_0-1} \ge 
\sum_{j}\sum_{a\in \cD_j}2^j \ge \frac{\# \cA \(\# \cB\)^2}{\#(\cA \cB)}$$
and the inequality $\# \cB\leq M$ gives
\begin{equation}
\label{A_1 bound}
N\#\cA_1\ge \frac{\# \cA\(\# \cB\)^2}{4\#(\cA \cB)\log{M}}.
\end{equation}

Since $1\leq M\leq t$, for any $x_1,x_2 \in \cX$ we have $x_1 \not\equiv x_2 \pmod t$ so that $g^{x_1}\not \equiv g^{x_2} \pmod p$, hence we get
\begin{equation}
\label{A bound}
\#\cA= M \Delta_1,
\end{equation}
\begin{equation}
\label{BB bound}
\#\cB = M \Delta_2,
\end{equation}
and
\begin{equation}
\label{AB bound}
\#(\cA \cB)= \# \{g^{x+y}~:~x\in \cX, y \in \cY \} \ll M.
\end{equation}
Inserting~\eqref{A bound},~\eqref{BB bound},~\eqref{AB bound} into~\eqref{A_1 bound} and recalling that $N\leq \# \cB$ and $\#\cA_1\leq \#\cA$ gives
\begin{equation}
\label{AN}
N\#\cA_1\gg M^2\frac{\Delta_1 \Delta_2^2}{\log{M}},
\end{equation}
\begin{equation}
\label{A_1 ub}
\#\cA_1\gg M\frac{\Delta_1\Delta_2}{\log{M}},
\end{equation}
\begin{equation}
\label{N ub}
N\gg M\frac{\Delta_2^2}{\log{M}}.
\end{equation}
By~\cite[Lemma~2.6]{BouGar} we have
$$\#(a\cA\pm a_0 \cA)\le \frac{\#(a\cA+(a\cB\cap a_0\cB))\#(a_0\cA+(a\cB\cap a_0\cB))}{\#(a\cB\cap a_0\cB))}\leq \frac{(\#(\cA+\cB))^2}{\#(a\cB \cap a_0\cB)},$$
so that for any $a \in \cA_1$, by~\eqref{B and a} 
\begin{equation}
\label{A+B}
\#(a\cA\pm a_0 \cA)\le \frac{(\#(\cA+\cB))^2}{N}.
\end{equation}
Using the same argument from the beginning of the proof, there exists $a_0' \in \cA_1$ such that
\begin{equation}
\label{aaaaa}
\sum_{a\in \cA_1}\#(a\cA_1 \cap a_0' \cA_1)\ge \frac{\(\#\cA_1\)^3}{\#\(\cA_1 \cA_1\)}.
\end{equation}
Let $\cA_2$ be the set of all $a \in \cA_1$ such that
\begin{equation}
\label{A_2}
\#((a/a_0')\cA_1\cap \cA_1) \ge \frac{\(\#\cA_1\)^2}{2\#\(\cA_1 \cA_1\)}.
\end{equation}
Then we have
\begin{equation}
\label{A_2 bound} 
\#\cA_2\ge \frac{\(\#\cA_1\)^2}{2\#\(\cA_1 \cA_1\)},
\end{equation}
since if the inequality~\eqref{A_2 bound} were false, we would have
\begin{align*}
\sum_{a\in \cA_1}\#(a\cA_1 \cap a_0' \cA_1)&=
\sum_{a\in \cA_2}\#(a\cA_1 \cap a_0' \cA_1)+\sum_{a\in \cA_1 \setminus \cA_2}\#(a\cA_1 \cap a_0' \cA_1)  \\
&\leq \#\cA_2\#\cA_1+\#\cA_1\frac{\(\#\cA_1\)^2}{2\#\(\cA_1 \cA_1\)} \\
& < \frac{\(\#\cA_1\)^3}{\#\(\cA_1 \cA_1\)}\left(\frac{1}{2}+\frac{1}{2}\right)=\frac{\(\#\cA_1\)^3}{\#\(\cA_1 \cA_1\)},
\end{align*}
which contradicts~\eqref{aaaaa}.
Let 
$$d_0=\max\{\ordp(a/a_0')~:~a \in \cA_2 \}=\ordp(a_0''/a_0')$$
for some $a_0'' \in \cA_2$.
We split the remaining proof into 2 cases:
 
{\it Case 1:\/}
$$\#\(\frac{\cA_1-\cA_1}{\cA_1-\cA_1}\)<\ordp(a_0''/a_0')$$
Let 
$$\cC=(a_0''/a_0)\cA_1 \cap \cA_1$$
then we have  
\begin{equation}
\label{C A_1 bound}
\#\(\frac{\cC-\cC}{\cA_1-\cA_1}\) \leq \#\(\frac{\cA_1-\cA_1}{\cA_1-\cA_1}\)<\ordp(a_0''/a_0')
\end{equation}
so there exists $c_1,c_2 \in \cC$ and $a_3, a_4 \in \cA_1$ such that $$(a_0'/a_0'')\frac{c_1-c_2}{a_3-a_4}\not\in  \frac{\cC-\cC}{\cA_1-\cA_1}.$$
Since if $(a_0'/a_0'')y \in  \frac{\cC-\cC}{\cA_1-\cA_1}$ for all $y \in \frac{\cC-\cC}{\cA_1-\cA_1}$ 
then the distinct elements 
$$y, (a_0'/a_0'')y,\ldots , (a_0'/a_0'')^{\ordp(a_0''/a_0')-1}y$$
all belong to $\frac{\cC-\cC}{\cA_1-\cA_1}$, contradicting~\eqref{C A_1 bound}. Using a similar argument, we may show  that we have strict subset inclusion $\cC \subset \cA_1$ so that we may choose $a_1, a_2 \in \cA_1$ such that 
$$\frac{a_1-a_2}{a_3-a_4}\not\in \frac{\cC-\cC}{\cA_1-\cA_1}.$$
Hence by~\cite[Lemma~3.1]{GlibKon} we have
$$\#((a_1-a_2)\cA+(a_3-a_4)\cA) \ge \# \left(\cC+\frac{a_1-a_2}{a_3-a_4}\cA_1\right)\ge \#\cA_1\#\cC$$
and since $a_0'' \in \cA_2$, we have by~\eqref{A_2}
\begin{equation}
\label{bound case 1}
\#((a_1-a_2)\cA+(a_3-a_4)\cA) \ge \frac{\(\#\cA_1\)^3}{\#\(\cA_1 \cA_1\)}.
\end{equation}
In~\cite[Lemma~2.7]{BouGar}  we take $k=4$ and 
$$B_1=a_1\cA, \quad B_2=-a_2\cA, \quad B_3=a_3\cA, \quad B_4=-a_4\cA, \quad  X=a_0\cA,$$
which gives
\begin{align}
\label{long inequality}
\#(a_1\cA-a_2\cA+a_3\cA &-a_4\cA)\leq \nonumber \\  &\frac{\#(a_0\cA+a_1\cA)\#(a_0\cA-a_2\cA)\#(a_0\cA+a_3\cA)\#(a_0\cA-a_4\cA)}{(\#\cA)^3}.
\end{align}
The inequality $\#\((a_1-a_2)\cA+(a_3-a_4)\cA\)\le\#\(a_1\cA-a_2\cA+a_3\cA-a_4\cA\)$ along with~\eqref{A+B} and~\eqref{bound case 1} gives
$$\(\#(\cA+\cB)\)^8\ge \frac{((\#\cA_1)^3(\#\cA)^3N^4}{\#\(\cA_1 \cA_1\)}.$$
Inserting~\eqref{A bound},~\eqref{AN} and~\eqref{N ub} into the above and using $\#\(\cA_1 \cA_1\)\ll M$ we get
\begin{equation}
\label{final 1}
\(\#(\cA+\cB)\)^8\ge M^{9+o(1)}\Delta_1^6\Delta_2^8.
\end{equation}

{\it Case 2:}
$$\#\(\frac{\cA_1-\cA_1}{\cA_1-\cA_1}\)\geq \ordp(a_0''/a_0')$$
 Then we have $M^4\ge \ordp(a_0''/a_0')$ and  writing  $a_0''=g^{x_0''}$ and $a_0'=g^{x_0}$, we have
$$\ordp(a_0''/a_0')=\frac{t}{\gcd(t,x_0''-x_0')}$$
 and since $1 \leq  |x_0''-x_0'| \leq M$ we get  $\ordp(a_0''/a_0')\gg M/t$. Combining this with the previous inequality gives $M\ge t^{1/5}$. We may suppose $\Delta_1 \Delta_2 \ge M^{-1/5}$ since otherwise the bound is trivial, so that~\eqref{A_1 ub} and~\eqref{A_2 bound} give
\begin{equation}
\label{A_2 ub}
\# \cA_2\gg M\frac{\Delta_1^2 \Delta_2^2}{\log^2{M}} \gg M^{3/5}  \gg 
t^{1/20}.
\end{equation}
Since $\cA_2 \subseteq \{ g^x~:~L_0+1\leq x \leq L_0+M \}$, we have
\begin{align*}
\#\cA_2=\sum_{a\in \cA_2}1\leq \sum_{\substack{a\in \cA_2 \\ \ordp(a/a_0')\leq d_0}}1& \le
\sum_{\substack{d \mid t \\ d\leq d_0}}\sum_{\substack{L_1+1\leq x \leq L_1+M \\ t \mid  dx}}1
\leq \left(\frac{Md_0}{t}+1\right)\tau(t)
\end{align*}
with $\tau(t)$ counting the number of divisors of $t$. By~\eqref{A_2 ub} and  the bound $\tau(t)\ll t^{o(1)}$
~\cite[Theorem~315]{HardyWright}
we obtain  $1 \ll |A_2|/\tau(t)$ and hence
\begin{equation}
\label{d_0 ub}
d_0\geq \frac{t}{M}\left(\frac{\#\cA_2}{\tau(t)}-1\right) \gg \frac{t \#\cA_2}{\tau(t)}M.
\end{equation}
By assumption on $d_0$ and~\eqref{A_2 bound} we have
$$
\#\( \frac{\cA_1-\cA_1}{\cA_1-\cA_1}\) \gg \frac{t\(\#\cA_1\)^2}{\tau(t)M^2}.
$$
Taking $G=\cA_1-\cA_1/\cA_1-\cA_1$ in~\cite[Lemma~3.3]{GlibKon} we see that there exists $\lambda \in (\cA_1-\cA_1)/(\cA_1-\cA_1)$ such that
$$\#\(\cA+\lambda \cA\) \ge  \#\(\cA_1+\lambda \cA_1\)
\ge \min\left \{ \(\#\cA_1\)^2, \frac{t\(\#\cA_1\)^2}{\tau(t)}M^2 \right \}.
$$
Hence there exist $a_1,a_2,a_3,a_4\in \cA_1$ such that 
$$\#\((a_1-a_2)\cA+(a_3-a_4)\cA\)\gg \(\#\cA_1\)^2$$
or
$$\#\((a_1-a_2)\cA+(a_3-a_4)\cA\)\gg \frac{t\(\#\cA_1\)^2}{\tau(t)}M^2.$$
For the first case, by~\eqref{A+B} and~\eqref{long inequality}
\begin{align*}
\(\#\cA_1\)^2 &\leq \frac{\#\( a_0\cA+a_1\cA\)\#\(a_0\cA-a_2\cA\)
\#\(a_0\cA+a_3\cA\)\#\(a_0\cA-a_4\cA\)}{\(\#\cA_1\)^3} \\
&\leq \frac{\(\#(\cA+\cB)\)^8}{( \#\cA)^3N^3}
\end{align*}
 and by~\eqref{A bound}, ~\eqref{AN} and~\eqref{N ub} we get
\begin{equation}
\label{final 2}
(\#(\cA+\cB))^8\gg M^{9+o(1)}\Delta_1^5\Delta_2^8\gg M^{9+o(1)}\Delta_1^6\Delta_2^8
\end{equation}
similarily for the second case, we get
$$\(\#(\cA+\cB)\)^8\gg \frac{t}{\tau(t)} M^{7+o(1)}\Delta_1^5\Delta_2^8$$
and recalling that $M\ge t^{1/5}$ and $\tau(t)\ll t^{o(1)}$, we may absorb the term $1/\tau(t)$ into
$M^{o(1)}$, which gives 
\begin{align}
\label{final 3}
\(\#(\cA+\cB)\)^8\gg tM^{7+o(1)}\Delta_1^5\Delta_2^8
\end{align}
and the result follows combining~\eqref{final 1},~\eqref{final 2} and~\eqref{final 3}. 
\end{proof}
Given $\cA,\cB \subset \F_p$, we write
$$\cE_{+}(\cA,\cB)=\# \{ (a_1,a_2,b_1,b_2) \in \cA^2 \times \cB^2 : a_1+b_1=a_2+b_2 \}.$$
Then we have~\cite[Lemma 7.1]{BouGar} 
\begin{lemma}
\label{lem:Bound 1}
Let $\cA, \cB \subset \F_p$, then
$$\left|\displaystyle\sum_{a\in \cA}\displaystyle\sum_{b\in \cB}e_p(xy)\right|^8 \le p(\# \cA)^4(\# \cB)^4\cE_{+}(\cA,\cA)\cE_{+}(\cB,\cB)$$
\end{lemma}
\begin{lemma}
\label{HBK1}
Suppose $g\in \F_p^*$ has order $t$ and let $\cA \subset F_p^*$ be the subgroup generated by $g$. Then for $N\leq t$ we have

\begin{equation*}
\label{bbb1}
\max_{\gcd(\lambda,p)=1}|S_{g,p}(\lambda,N)|\ll \begin{cases} p^{1/8}\cE_{+}(\cA,\cA)^{1/4}\log{t} \\ p^{1/4}t^{-1/4}\cE_{+}(\cA,\cA)^{1/4}\log{t}.
\end{cases}
\end{equation*}
\end{lemma}
\begin{proof}
Let
$$\sigma(a,c)=\sum_{n=1}^{t}e_t(an)e_p(cg^n)$$
then we have
$$S_{g,p}(\lambda,N)=\displaystyle\sum_{n=1}^{N}e_p(\lambda g^n)=\frac{1}{t}\displaystyle\sum_{k=1}^{t}\displaystyle\sum_{j=1}^{N}e_t(-kj)\displaystyle\sum_{n=0}^{t-1}e_t(kn)e_p(\lambda g^{n})$$
so that
\begin{align}
\label{CI}
|S_{g,p}(\lambda,N)|&\leq \frac{1}{t}\sum_{k=1}^{t}\left|\sum_{j=1}^{N}e_t(-kj)\right|\max_{k\in \F_p} \left|\sigma(k,\lambda)\right| \nonumber \\
&\ll \log{t}\max_{k\in \F_p} \left|\sigma(k,\lambda)\right|.
\end{align}
By~\cite[Lemma 3.14]{Sh} for any integers $k,\lambda,$ with $\gcd(\lambda,p)=1$, we have
$$|\sigma(k,\lambda)|\le p^{1/4}t^{-1/4}\cE_{+}(\cA,\cA)^{1/4},$$
$$|\sigma(k,\lambda)|\le p^{1/8}\cE_{+}(\cA,\cA)^{1/4}$$
and the result follows combining  these bounds with~\eqref{CI}. 
\end{proof}
\section{Proof of Theorem 1}

Let $J(g,N)$ equal the number of solutions to the equation
$$g^{x_1}+g^{x_2}=g^{x_3}+g^{x_4}, \qquad 1\leq x_1,x_2,x_3,x_4 \leq N,$$
then we have
\begin{equation}
\label{J}
\sum_{\lambda \in  \F_p^*}\left|S_{\lambda}(p;g,N)\right|^4 \leq 
\sum_{\lambda \in  \F_p}\left|S_{\lambda}(p;g,N)\right|^4=pJ(g,N).
\end{equation}
 Given $\cA,\cB \subseteq \F_p$ and $\cE_0 \subset \cA \times \cB$ we write
$$\cA{+}_{\cE_0} \cB=\{ a+b : (a,b)\in \cE_0 \}$$
so that by ~\cite[Lemma~2.4]{BouGar}  there exists $\cE_0\subseteq \cA\times \cA$ such that
 $$\cE_{+}(\cA,\cA)\leq \frac{8(\#\cE_0)^2}{\#(\cA {+}_{\cE_0} \cA)}\log^2(e\# \cA)$$
and writing $K=N^2/\# \cE_0$ gives
\begin{equation}
\label{mult energy}
J(g,N) \leq \frac{N^{4+o(1)}}{\#\(\cA \mathop{+}_{\cE_0} \cA\) K^2}.
\end{equation}
Since $N\leq t$ we have $\#\cA =N$ so that $\# \cE_0= (\#\cA)^2/K$.  Hence by~\cite[Lemma~2.3]{BouGar}
there exists $\cA_1,\cA_2\subseteq \cA$ and integer $Q$ with 
\begin{equation}
\label{subsets}
\#\cA_1 \gg \frac{N}{K}, \qquad \#\cA_2 \gg \frac{N^2}{QK^2\log{N}},
\end{equation}
such that
\begin{equation}
\label{2.3}
\(\#\(\cA {+}_{\cE_0} \cA\)\)^3\gg \#\(\cA_1+\cA_2\)\frac{QN}{K^3\log{N}}.
\end{equation}
By~\eqref{subsets} and Lemma \ref{sum} we have
\begin{equation}
\begin{split}
\label{almost finished 11}
\#\(\cA_1+\cA_2\) &> \min\left \{N^{9/8+o(1)}\frac{1}{K^{3/4}}\frac{N}{QK^2}, t^{1/8}N^{7/8+o(1)}\frac{1}{K^{5/8}}\frac{N}{QK^2} \right \} \\
&\ge \frac{t^{1/8}N^{4+o(1)}}{t^2K^{5+3/8}}\frac{QK^{2+3/4}}{N^{1+7/8+o(1)}}\left(\frac{1}{N^{2/8+o(1)}+t^{1/8}K^{1/8}}\right)  \\
& \ge \frac{t^{1/8}N^{3-7/8+o(1)}}{QK^{19/8}}\left(\frac{1}{N^{2/8+o(1)}+t^{1/8}K^{1/8}}\right)
\end{split}
\end{equation} 
and from~\eqref{2.3} and~\eqref{almost finished 11} we get
\begin{equation}
\label{last sub}
\(\#\(\cA{+}_{\cE_0} \cA\)\)^{-1} < \frac{K^{36/24}}{N^{1+1/24+o(1)}}
\left ( \left(\frac{N^2}{Kt}\right)^{1/24}+1 \right).
\end{equation}
Combining~\eqref{mult energy} with~\eqref{last sub} gives
$$J(g,N)<K^{36/24-2}N ^{3-1/24}\left(  \left(\frac{N ^2}{Kt}\right)^{1/24}+1 \right)<N ^{3-1/24}\left(  \left(\frac{N ^2}{t}\right)^{1/24}+1 \right)$$
and since $K\ge 1$ the result follows.  
\section{Proof of Theorem 2}
We follow the method of \cite{KoSh} and begin with considering
$$\sigma_{p,g}(N)=\max_{1\leq K \leq N}\max_{\gcd(\lambda,p)=1} |S_{p,g}(\lambda,K)|$$
so that for any integer $K$ we have
$$\left|S_{p,g}(\lambda,N)-\frac{1}{K}\displaystyle\sum_{k=1}^{K}\displaystyle\sum_{n=1}^{N}e_p{(\lambda g^{k+n})}\right|\leq 2\sigma_{g,p}(K).$$
 Taking  $\cA=\{ g^{n} : 1\leq n \leq N \}$, $\cB=\{ \lambda g^{n} : 1\leq n \leq K \}$ in Lemma~\ref{lem:Bound 1}, we have by Theorem~\ref{lem:Bound SN4}
\begin{align*}
&\left|\frac{1}{K}\displaystyle\sum_{k=1}^{K}\displaystyle\sum_{n=1}^{N}e_p(\lambda g^{k+n})\right|
\leq \\ & \quad \quad \quad   p^{1/8}N^{167/192+o(1)}\left(1+\left(\frac{N^2}{t}\right)^{1/192}\right)K^{-25/192+o(1)}\left(1+\left(\frac{K^2}{t}\right)^{1/192}\right)
\end{align*}
and letting $K=\lfloor{N/3}\rfloor$ we get
$$\sigma_{p,g}(N)\leq \sigma_{p,g}(\lfloor{N/3}\rfloor)+
p^{1/8}N^{71/96+o(1)}\left(1+\left(\frac{N^2}{t}\right)^{1/96}\right).$$
Repeating the above argument recursively, we end up with $O(\log{N})$ terms all bounded by 
$$p^{1/8}N^{71/96}\left(1+\left(\frac{N^2}{t}\right)^{1/96}\right)$$
which gives
\begin{align}
\label{bbound 1}
\max_{\gcd(\lambda,p)=1}|S_{g,p}(\lambda,N)|\leq p^{1/8}N^{71/96+o(1)}\left(1+\left(\frac{N^2}{t}\right)^{1/96}\right).
\end{align}
Also, we have from H\"{o}lder's inequality,
\begin{align*}
\left|\displaystyle\sum_{k=1}^{K}\displaystyle\sum_{n=1}^{N}e_p(\lambda g^{k+n})\right|&\leq K^3
\displaystyle\sum_{k=1}^{K}\left|\displaystyle\sum_{n=1}^{N}e_p(\lambda g^{k+n})\right|^4
\leq \displaystyle\sum_{a \in \F_p}|S_{p,g}(a,N)|^4
\end{align*}
so by Theorem~\ref{lem:Bound SN4} we get
\begin{align*}
\left|\displaystyle\sum_{k=1}^{K}\displaystyle\sum_{n=1}^{N}e_p(\lambda g^{k+n})\right|
\leq p^{1/4}K^{-1/4+o(1)}N^{71/96+o(1)}\left(1+\left(\frac{N^2}{t}\right)^{1/96}\right)
\end{align*}
and taking $K=\lfloor{N/3}\rfloor$ gives
$$\sigma_{p,g}(N)\leq \sigma_{p,g}(\lfloor{N/3}\rfloor)+ p^{1/4}N^{47/96+o(1)}\left(1+\left(\frac{N^2}{t}\right)^{1/96}\right).$$
As before we end up with the bound
\begin{equation}
\label{bbound 2}
\max_{\gcd(\lambda,p)=1}|S_{g,p}(\lambda,N)|\leq p^{1/4}N^{47/96+o(1)}\left(1+\left(\frac{N^2}{t}\right)^{1/96}\right)
\end{equation}
and the result follows combining~\eqref{bbound 1} and~\eqref{bbound 2}.

\section{Proof of Theorem 3}
Let $\cA \subset \F_p^{*}$ be the subgroup generated by $g$, so by~\cite[Theorem~34]{Sk} we have
\begin{equation}
\label{BE1}
\cE_{+}(\cA,\cA)\ll \begin{cases} t^{22/9}(\log{p})^{2/3}, \quad \text{if} \ \ t\le p^{3/5}(\log{p})^{-6/5}, \\ t^3p^{-1/3}(\log{p})^{4/3}, \quad \text{if} \ \ t>p^{3/5}(\log{p})^{-6/5}. \end{cases}
\end{equation}
We consider first when $t\leq p^{1/2}$. Combining Lemma~\ref{HBK1} with~\eqref{BE1} gives
$$\max_{\gcd(\lambda,p)=1}|S_{g,p}(\lambda,N)|\leq p^{1/8}t^{22/36}(\log{p})^{7/6}.$$
For $p^{1/2}<t\le p^{3/5}(\log{p})^{-6/5}$ we have,
$$\max_{\gcd(\lambda,p)=1}|S_{g,p}(\lambda,N)|\leq p^{1/4}t^{13/36}(\log{p})^{7/6}.$$
If  $p^{3/5}(\log{p})^{-6/5}< t \le p^{2/3}(\log{p})^{-2/3}$
$$\max_{\gcd(\lambda,p)=1}|S_{g,p}(\lambda,N)|\leq p^{1/6}t^{1/2}(\log{p})^{4/3}$$
and for $p^{2/3}(\log{p})^{-2/3}<t$, from~\cite[Lemma 2]{Kor}
$$\max_{\gcd(\lambda,p)=1}|S_{g,p}(\lambda,N)|\leq p^{1/2}\log{p}$$
and the result follows combining the above bounds.

\end{document}